\documentclass[a4paper,fleqn]{cas-sc}

\usepackage[numbers,sort&compress]{natbib}

\usepackage{amsmath,amsthm,amsfonts,amssymb,amsopn,extarrows,mathtools}

\usepackage{hyperref,cleveref}

\usepackage{tikz}
\usetikzlibrary{arrows.meta,positioning,calc}

\usepackage{xcolor,soul}

\def\tsc#1{\csdef{#1}{\textsc{\lowercase{#1}}\xspace}}
\tsc{WGM}
\tsc{QE}
\tsc{EP}
\tsc{PMS}
\tsc{BEC}
\tsc{DE}

\DeclareMathOperator{\rp}{Re}
\DeclareMathOperator{\ip}{Im}
\DeclareMathOperator{\tr}{tr}
\DeclareMathOperator{\spec}{Spec}
\DeclareMathOperator{\Deg}{Deg}
\newcommand{\iu}{\boldsymbol{\mathrm{i}}}
\newcommand{\R}{\mathbb{R}}

\newcommand{\tp}{\mathrm{T}}

\newtheorem{theorem}{Theorem}[section]
\newtheorem{lemma}[theorem]{Lemma}

\newtheorem{corollary}[theorem]{Corollary}

\theoremstyle{definition}

\begin{document}
\let\WriteBookmarks\relax
\def\floatpagepagefraction{1}
\def\textpagefraction{.001}
\shorttitle{New Bounds for the Spectral Radius and Low Energy of the $A_\alpha$-Matrix of Digraphs}
\shortauthors{S. Huang}

\title[mode=title]{New Bounds for the Spectral Radius and Low Energy of the $A_\alpha$-Matrix of Digraphs}



\author[1]{Silin Huang}[orcid=0009-0004-8963-6394]
\fnmark[1]
\ead{huangsilin@stu.jnu.edu.cn}

\affiliation[1]{organization={College of Information Science and Technology, Jinan University},
                city={Guangzhou},
                citysep={}, 
                postcode={510632}, 
                state={Guangdong},
                country={China}}


\begin{abstract}
The $A_\alpha$-matrix of a digraph $D$
is defined as a linear convex combination
$\alpha\Deg(D)+(1-\alpha)A(D)$
of the adjacency matrix $A(D)$
and the diagonal out-degree matrix $\Deg(D)$,
where $\alpha\in[0,1]$.
The low energy of $A_\alpha(D)$ is defined as the sum of
the absolute values of the real parts of the eigenvalues of $A_\alpha(D)$.
In this paper, we establish new upper bounds
for the spectral radius of the $A_\alpha$-matrix
and derive two Koolen--Moulton type upper bounds for its low energy,
together with characterizations of the equality cases.
Numerical comparisons further show that
these bounds can be sharper than existing bounds
for certain digraph families.
Furthermore, when $\alpha=0$,
our results recover several classical bounds,
and in particular, the low-energy bounds generalizes
the classical Koolen--Moulton bound.
\end{abstract}




\maketitle

\section{Introduction}
A \emph{directed graph} (or simply a \emph{digraph}) $D=(V,E)$
consists of a finite vertex set $V$
and an arc set $E\subseteq V\times V$.
Throughout this paper, we consider simple digraphs,
i.e., digraphs without self-loops.
Let $n=|V|$ and $m=|E|$.
For a vertex $v\in V$, let $d^+(v)$ denote its out-degree.
Then the adjacency matrix of $D$ is the $n\times n$ matrix
$A(D)=(a_{uv})$ with $a_{uv}=1$ if $u\to v$ is an arc
and $a_{uv}=0$ otherwise, and
$\Deg(D)=\mathrm{diag}(d^+(v_1),\dots,d^+(v_n))$
is the diagonal matrix of vertex out-degrees.

For $\alpha\in[0,1]$, the \emph{$A_\alpha$-matrix}
of a digraph $D$ is defined by
\begin{equation}
  A_\alpha(D)=\alpha\Deg(D)+(1-\alpha)A(D).
\end{equation}
This matrix is first introduced by Nikiforov~\cite{Nikiforov2017}
for simple undirected graphs
and was subsequently generalized to digraphs by Liu et al.~\cite{Liu2019}.
It unifies several classical matrices, as
$A_0(D)=A(D)$, $2A_{1/2}(D)=Q(D)$
(the signless Laplacian matrix),
$(A_{\alpha_1}(D)-A_{\alpha_2}(D))/(\alpha_1-\alpha_2)=L(D)\,(\alpha_1\ne\alpha_2)$
(the Laplacian matrix),
and $A_1(D)=\Deg(D)$.
It is also called the \emph{generalized adjacency matrix}
in several literature, e.g.,~\cite{Bagipur2022}.

In the past few years, extensive research
has been done on the spectral radius of the $A_\alpha$-matrix
of undirected graphs.
Many bounds on the spectral radius
have been established~\cite{Xue2018,Lin2021,Alhevaz2024}.
Furthermore, the spectral radius is used to determine
the existence of certain subgraph structures~\cite{Zhou2024,Lv2025,Zhang2025}.

There have also been a noticeable amount of research
on the spectral radius of $A_\alpha$-matrix of digraphs.
For instance, Baghipur et al.~\cite{Bagipur2022} gave some of
its sharp upper and lower bounds in terms of parameters such as
the number $n$ of vertices, the number $m$ of arcs,
the out-degrees, the maximum out-degree, and
the second maximum out-degree.
Carmona and Ganie~\cite{Carmona2024} derived some of
its sharp upper and lower bounds in terms of parameters such as
$n,\,m$, and the average $2$-out- and $2$-in-degrees.
Following this line of research,
in this paper, we establish some new sharp upper and lower bounds
in terms of $n,\,m$, the first out-degree Zagreb index $Z$,
and the number $c_2$ of closed walks of length $2$.

Independently, the concept of \textit{graph energy},
which plays an important role in chemical graph theory,
was introduced by Gutman~\cite{Gutman1978}
as the sum of the absolute values
of the eigenvalues of the adjacency matrix.
Since then, numerous variants of definitions of graph energy
have been developed.
For instance, Guo and Zhou~\cite{Guo2020}
introduced the definition of
$E(D)=\sum_{i=1}^n\left|\lambda_i-2\alpha m/n\right|$
for the $A_\alpha$-matrix of undirected graphs;
Xi~\cite{Xi2021} proposed the definition of
$E(D)=\sum_{i=1}^n\lambda_i^2$
for the $A_\alpha$-matrix of digraphs;
Pe\~{n}a and Rada~\cite{PenaRada2008} defined
$E(D)=\sum_{i=1}^n\left|\rp(\lambda_i)\right|$
for the adjacency matrix of digraphs,
which was later extended to the $A_\alpha$-matrix.
Each definition has its own motivation and advantages,
which we will not elaborate on here.
In this paper, we adopt the definition of Pe\~{n}a and Rada's,
\begin{equation}
  E(D)=\sum_{i=1}^n\left|\rp(\lambda_i)\right|,
  \qquad \lambda_i\in\spec(A_\alpha(D)),\,i=1,2,\dots,n,
\end{equation}
which is also called the \emph{low energy}
in Brauldi's survey~\cite{Brauldi2010}.
Remark that, in the case of symmetric digraphs,
which are equivalent to undirected graphs,
this reduces to the classical Gutman-type graph energy form,
which is one of the reasons that this definition is natural.

The energy of the $A_\alpha$-matrix of undirected graphs
has been much studied;
readers may refer to~\cite{Pirzada2021,Yang2024}.
As for the low energy of the $A_\alpha$-matrix of digraphs,
Pirzada and Mushtaq~\cite{Pirzada2025} investigated the low energy
of some specific classes of digraphs
and obtained an upper bound for the low energy of regular digraphs.
Following this line of research,
we obtain two new upper bounds for the low energy of any digraph,
which is more general and even improve existing bounds in~\cite{Pirzada2025}.
We also completely characterize the equality cases.

It is worth mentioning that for $\alpha=0$,
our bounds for spectral radius and low energy
reduce to several classical bounds for the adjacency matrix.
Specifically, the low energy bounds we obtained
may be viewed as generalization of the classical Koolen--Moulton bound.

This paper is organized as follows.
Section~\ref{sec:2}
introduces notation, lists some preliminary facts,
and investigate the spectra of some special digraph families
that will be useful afterwards.
In Section~\ref{sec:3},
we derive several upper and lower bounds
for the $A_\alpha$-matrix spectral radius
and characterize the extremal graphs.
Section~\ref{sec:4}
is devoted to Koolen--Moulton type bounds
for the $A_\alpha$-matrix low energy.
In Section~\ref{sec:5},
we introduce some classes of digraphs
which achieve equality in the Koolen--Moulton type bounds
in Section~\ref{sec:4}.

\section{Preliminaries}\label{sec:2}
\subsection{Notations and Definitions}
Throughout the paper, our discussion focuses on
directed graphs~(digraphs) without self-loops.
For a digraph $D(V,E)$, we adopt the following notation:

\begin{itemize}
  \item $n=|V|$, its order; $m=|E|$, its size;
  $d^+(v)$, the out-degree of a vertex $v$;
  $d^-(v)$, the in-degree of a vertex $v$;
  $N^+(u)=\{v\in V:u\to v\}$ and $N^-(u)=\{v\in V:v\to u\}$,
  the out-neighborhood and in-neighborhood of vertex $u$, respectively;

  \item $A(D)$, its adjacency matrix;
  $\Deg(D)$, the diagonal matrix of vertex out-degrees of $D$;
  $A_\alpha(D)=\alpha\Deg(D)+(1-\alpha)A(D)$, its $A_\alpha$-matrix;

  \item $M_k=\tr(A_\alpha(D)^k)$, the $k$-th spectral moment of $A_\alpha(D)$;
  $Z=Z_1^+=\sum_{v\in V}(d^+(v))^2$,
  the (First Out-degree) Zagreb index of $D$;
  $c_2=\sum_{u\in V}\sum_{v\in V:u\sim v} A(D)_{u,v}A(D)_{v,u}$,
  the number of closed walks of length $2$ in $D$,
  which equals to twice the number of digons of $D$,
  where $u\sim v$ means there exists arc $u\to v$ and/or $v\to u$;

  \item $\lambda_1,\lambda_2,\dots,\lambda_n$,
  the eigenvalues of $A_\alpha(D)$;
  $(x_i,y_i)=(\rp(\lambda_i),\ip(\lambda_i))$,
  the real part and imaginary part of the eigenvalue $\lambda_i$;
  $\spec(D)$, the spectrum of $A_\alpha(D)$;
  $\rho(D)$, the spectral radius of $A_\alpha(D)$;
  $E(D)$, the low energy of $A_\alpha(D)$.
\end{itemize}

\subsection{Basic Spectral Properties of the $A_\alpha$-Matrix}
We give the expressions for the first and second
spectral moment and the Frobenius norm of $A_\alpha$-matrix,
which will be frequently used for the rest of the paper.

\begin{lemma}\label{lem:m1}
  For a digraph $D$, we have $M_1=\alpha m$.
\end{lemma}
\begin{proof}
  By direct calculation, we have
  \begin{equation}
    M_1=\tr(A_\alpha)
    =\tr(\alpha\Deg+(1-\alpha)A)
    =\alpha\underbrace{\tr(\Deg)}_{m}+(1-\alpha)\underbrace{\tr(A)}_{0}
    =\alpha m.
  \end{equation}
\end{proof}

\begin{lemma}\label{lem:m2}
  For a digraph $D$, we have $M_2=\alpha^2 Z+(1-\alpha)^2 c_2$.
\end{lemma}
\begin{proof}
  This comes directly from
  \begin{align}
    M_2&=\tr(A_\alpha^2)=\tr([\alpha\Deg+(1-\alpha)A]^2)
    =\tr(\alpha^2\Deg^2+(1-\alpha)^2 A^2+\alpha(1-\alpha)(\Deg A+A\Deg)) \\
    &=\alpha^2\tr(\Deg^2)+(1-\alpha)^2\tr(A^2)+2\alpha(1-\alpha)\tr(\Deg A),
  \end{align}
  where
  \begin{equation}
    \tr(\Deg^2)=\sum_{i=1}^n (d_i^+)^2=Z,
    \qquad
    \tr(A^2)=\sum_{u\in V}\sum_{v\in V:u\sim v} A_{u,v}A_{v,u}=c_2,
    \qquad\text{and}\qquad
    \tr(\Deg A)=0.
  \end{equation}
\end{proof}

\begin{lemma}
  The Frobenius norm of $A_\alpha$ is given by
  \begin{equation}
    \|A_\alpha\|_F=\sqrt{\sum_{u,v\in V}|(A_\alpha)_{uv}|^2}
    =\sqrt{\alpha^2 Z+(1-\alpha)^2 m}.
  \end{equation}
\end{lemma}
\begin{proof}
  This can be obtained by direct calculation.
\end{proof}

\subsection{Spectra of Two Special Families of Digraphs}
\begin{lemma}\label{lem:spectrum-complete-digraph}
  Let $D$ be a digraph of order $n$.
  Then the spectrum of $A_\alpha=A_\alpha(D)$ is of the form
  $\{n-1,(\alpha n-1)^{(n-1)}\}$
  if and only if $D=\overleftrightarrow{K_n}$
  is the complete symmetric digraph on $n$ vertices.
\end{lemma}
\begin{proof}[Proof of Necessity ($\implies$).]
  \vspace{0.5em}\noindent\textbf{Case $\alpha=0$.}
  On one hand, the trace of $A_\alpha^2$ is given by
  \begin{equation}
    \tr(A_\alpha^2)=\sum_{i=1}^n\lambda_i^2=(n-1)^2+(n-1)(0n-1)^2=n(n-1).
  \end{equation}
  On the other hand, by Lemma~\ref{lem:m2},
  $\tr(A_\alpha^2)=c_2$.
  Therefore $c_2=n(n-1)$,
  implying that there are $\frac{c_2}{2}=\frac{n(n-1)}{2}=\binom{n}{2}$
  digons in $D$.
  Such a digraph is uniquely
  the complete symmetric digraph $\overleftrightarrow{K_n}$.

  \vspace{0.5em}\noindent\textbf{Case $\alpha\in(0,1]$.}
  On one hand, the trace of $A_\alpha$ is given by
  $\tr(A_\alpha)=\sum_{i=1}^n\lambda_i=(n-1)+(n-1)(\alpha n-1)=\alpha n(n-1)$.
  On the other hand, by definition,
  $\tr(A_\alpha)=\sum_{v\in V}\alpha d_v^+=\alpha m$.
  Equating the two yields $\alpha m=\alpha n(n-1)$.
  Since $\alpha\ne 0$, we have $m=n(n-1)$.
  A simple digraph with $n$ vertices and $n(n-1)$ arcs
  is uniquely the complete symmetric digraph $\overleftrightarrow{K_n}$.
\end{proof}
\begin{proof}[Proof of Sufficiency ($\impliedby$).]
  This is easy by direct calculation.
\end{proof}

\begin{lemma}\label{lem:spectrum-complete-digraph-gen}
  Let $D$ be a digraph of order $n$.
  Suppose $\alpha\in[0,1)$ and $A_\alpha=A_\alpha(D)$ is a normal matrix.
  Then the spectrum of $A_\alpha(D)$ is of the form
  $\{\rho,\lambda,\dots,\lambda\}$ with $\rho\ne\lambda$
  if and only if (i) $D\cong\overleftrightarrow{K_n}$;
  or (ii) $D\cong\overleftrightarrow{K_k}\cup(n-k)K_1$ and $\alpha=1/k$.
\end{lemma}
\begin{proof}[Proof of Necessity ($\implies$).]
  Since $A_\alpha$ is real, normal and has real spectrum,
  $A_\alpha$ is real symmetric
  and $D$ is a symmetric digraph.

  Let $\mathbf u$ be a unit eigenvector of $A_\alpha$
  for the eigenvalue $\rho$.
  Since $A_\alpha$ is symmetric,
  we have the spectral decomposition
  \begin{equation}\label{eq:two-eigs-decomp}
    A_\alpha=(\rho-\lambda)\mathbf u\mathbf u^\tp+\lambda I.
  \end{equation}
  Moreover, since $A_\alpha$ is entrywise nonnegative,
  $\mathbf u$ can be chosen entrywise nonnegative.

  For $i\ne j$, \eqref{eq:two-eigs-decomp} gives
  \begin{equation}\label{eq:offdiag}
    (A_\alpha)_{ij}=(\rho-\lambda)u_i u_j.
  \end{equation}
  Remark that $(A_\alpha)_{ij}\in\{0,1-\alpha\}$ for $i\ne j$.
  Let $S=\{i:u_i>0\}$ and put $k=|S|$.

  We claim that the induced subdigraph on $S$ is $\overleftrightarrow{K_k}$,
  there are no arcs between $S$ and $V\setminus S$ (if not empty),
  and every vertex in $V\setminus S$ (if any) is isolated.
  By~\eqref{eq:offdiag}, if $i,j\in S$ and $i\ne j$,
  then $u_i u_j>0\implies(A_\alpha)_{ij}\ne 0$,
  forcing $i$ and $j$ to form a digon.
  Thus $D[S]\cong \overleftrightarrow{K_k}$.
  If $i\in S$ and $j\notin S$, then
  $u_j=0\implies(A_\alpha)_{ij}=0$.
  So there are no arcs between $S$ and $V\setminus S$.
  If $i,j\notin S$, then
  $u_i=u_j=0\implies(A_\alpha)_{ij}=0$.
  Therefore every vertex in $V\setminus S$ is isolated.

  We now consider whether $D$ is strongly connected.

  \vspace{0.5em}\noindent\textbf{Case 1: $D$ is strongly connected.}
  Then $A$ is irreducible.
  Since $\Deg$ is non-negative,
  $A_\alpha=\alpha\Deg+(1-\alpha)A$ is also irreducible.
  hence by Theorem~\ref{thm:p-f-thm},
  $\mathbf u$ is entrywise positive.
  Thus $S=V$ and $k=n$, so $D=\overleftrightarrow{K_n}$.

  \vspace{0.5em}\noindent\textbf{Case 2: $D$ is not strongly connected.}
  Then $k<n$ and there exists $j\notin S$.
  Since $j$ is isolated, $d^+(j)=0$ and $(A_\alpha)_{jj}=\alpha d^+(j)=0$.
  On the other hand, from~\eqref{eq:two-eigs-decomp} and $u_j=0$
  we have $(A_\alpha)_{jj}=\lambda$. Hence $\lambda=0$.

  Take $i\in S$. Since $D[S]\cong \overleftrightarrow{K_k}$,
  we have $d^+(i)=k-1$,
  and for $i\ne\ell$ in $S$, $(A_\alpha)_{i\ell}=1-\alpha$.
  But, by~\eqref{eq:offdiag} and $\lambda=0$,
  we have $(A_\alpha)_{i\ell}=\rho u_i u_\ell$.
  This forces $u_i u_\ell$ to be constant over all distinct $i,\ell\in S$,
  hence all $u_i$ are equal, say $u_i=t>0$.
  Comparing diagonal and off-diagonal entries for vertices in $S$, we have
  \begin{equation}
    (A_\alpha)_{ii}=\alpha(k-1)=\rho t^2,\qquad
    (A_\alpha)_{i\ell}=1-\alpha=\rho t^2\quad(\ell\in S,i\ne\ell).
  \end{equation}
  Therefore $\alpha(k-1)=1-\alpha\implies\alpha=1/k$.
  This gives $\spec(A_\alpha)=\{k-1,0^{(n-1)}\}$
  and completes the proof.
\end{proof}
\begin{proof}[Proof of Sufficiency ($\impliedby$).]
  This is easy by direct calculation.
\end{proof}

\begin{lemma}\label{lem:spectrum-k2}
  Let $D$ be a digraph of order $n$ and let $A=A(D)$.
  Then the spectrum of $A$ is of the form
  $\{\pm\lambda,\dots,\pm\lambda\}$
  for some $\lambda\in(0,+\infty)$
  if and only if $\lambda=1$ and every strong component of $D$
  is a digon $\overleftrightarrow{K_2}$.
  Equivalently, $D$ can be obtained from a disjoint union of digons
  by adding (possibly none) arcs between distinct digons.
\end{lemma}
\begin{proof}[Proof of Necessity ($\implies$).]
  Suppose $\spec(A)=\{\pm\lambda,\dots,\pm\lambda\}$ for some $\lambda>0$.
  Let $n_+$ (resp.\ $n_-$) be the multiplicity of $\lambda$ (resp.\ $-\lambda$) in the spectrum.
  Since $A$ has zero diagonal, $\tr(A)=0$.
  Hence
  \begin{equation}
    0=\tr(A)=\sum_{i=1}^n \lambda_i=(n_+-n_-)\lambda,
  \end{equation}
  which implies $n_+=n_-=n/2$ and $n$ is even.

  Let $D_1,\dots,D_t$ be the strong components of $D$.
  Then there exists a permutation matrix $P$
  such that $P^\tp A P$ is block upper triangular
  with irreducible diagonal blocks
  $A_1,\dots,A_t$, where $A_i=A(D_i)$.
  Since eigenvalues of a block upper triangular matrix
  are precisely the union of eigenvalues of its diagonal blocks,
  we have that for each $i$,
  $\spec(A_i)\subseteq\{\lambda,-\lambda\}$.

  Fix $i$. Since $A_i$ is nonnegative and irreducible,
  by Theorem~\ref{thm:p-f-thm}, its spectral radius $\rho(A_i)$
  is an eigenvalue with algebraic multiplicity $1$.
  Therefore we must have $\rho(A_i)=\lambda$,
  and the eigenvalue $\lambda$
  of $\spec(A_i)$ has multiplicity $1$.
  On the other hand, since $\tr(A_i)=0$,
  it holds that $\spec(A_i)=\{\lambda,-\lambda\}$.

  Therefore every strong component $D_i$ has exactly two vertices.
  As $D_i$ is strongly connected,
  it must consist of the two opposite arcs,
  i.e., $D_i\cong \overleftrightarrow{K_2}$.
  The adjacency matrix of $\overleftrightarrow{K_2}$ is
  $\begin{pmatrix}0&1\\1&0\end{pmatrix}$, whose eigenvalues are $\pm 1$.
  Hence $\lambda=1$, completing the proof.
\end{proof}
\begin{proof}[Proof of Sufficiency ($\impliedby$).]
  This is straightforward by reversing the steps above.
\end{proof}

\begin{figure}[htbp]\centering
  \begin{tikzpicture}[
    >={Stealth[round, sep]},
    node distance=1.5cm,
    vertex/.style={
        circle,
        draw=teal!80!black,
        fill=teal!10,
        very thick,
        minimum size=20pt,
        inner sep=0pt,
        font=\small\bfseries\sffamily
    },
    arc/.style={
        ->,
        draw=teal!60,
        thick,
        bend left=15
    },
    inter/.style={
        ->,
        draw=teal!80!black,
        thick,
        bend left=25,
        shorten <=2pt,
        shorten >=2pt
    }
  ]

  \def\n{3}
  \def\hsep{3.5}

  \pgfmathtruncatemacro{\nminusone}{\n-1}

  \foreach \i in {1,...,\n} {

      \pgfmathsetmacro{\xcoord}{(\i-1)*\hsep}

      \pgfmathsetmacro{\idxA}{int(2*\i-1)}
      \node[vertex] (u\i) at (\xcoord, 0) {$v_{\idxA}$};

      \pgfmathsetmacro{\idxB}{int(2*\i)}
      \node[vertex] (v\i) at (\xcoord + 1.8, 0) {$v_{\idxB}$};

      \draw[arc] (u\i) to (v\i);
      \draw[arc] (v\i) to (u\i);

      \node[below=0.8cm of $(u\i)!0.5!(v\i)$,
            font=\footnotesize\sffamily, text=teal!80!black]
          {$(\overleftrightarrow{K_2})_{\i}$};
  }

  \foreach \i [evaluate=\i as \j using int(\i+1)] in {1,...,\nminusone} {
      \draw[inter] (v\i) to (u\j);
  }

  \end{tikzpicture}
  \caption{A digraph obtained from $3$ digons
  by adding arcs between distinct digons.}
  \label{fig:3-digons-strong-components}
\end{figure}
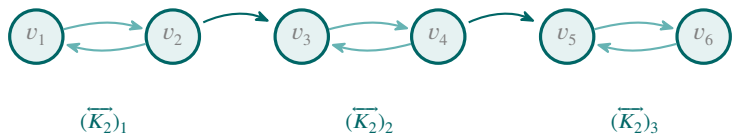

\subsection{Perron--Frobenius Theorem}
Perron--Frobenius Theorem is an important theorem
in matrix theory that has wide applications in many fields
of mathematics and will be used frequently in this paper.

\begin{theorem}[Perron--Frobenius Theorem]\label{thm:p-f-thm}
  Let $M$ be an $n\times n$ entrywise nonnegative matrix,
  and let $\rho=\rho(M)$ denote its spectral radius.
  Then the following hold.
  \begin{itemize}
    \item[(i)] $\rho$ is an eigenvalue of $M$.
    Moreover, there exists a nonnegative eigenvector
    $\mathbf{x}\ne\mathbf{0}$, such that
    $M\mathbf{x}=\rho\mathbf{x}$.
    \item[(ii)] If $M$ is irreducible,
    then $\rho$ is a simple eigenvalue of $M$
    (i.e., it has algebraic multiplicity $1$),
    and there exists a positive eigenvector
    $\mathbf{x}$ such that
    $M\mathbf{x}=\rho\mathbf{x}$.
    \item[(iii)] If $M$ is primitive,
    i.e., $M^k>0$ for some integer $k\ge 1$,
    then $\rho(M)$ is the unique eigenvalue of $M$
    on $\{\lambda:|\lambda|=\rho(M)\}$.
    Equivalently, every eigenvalue
    $\lambda\ne\rho$ of $M$ must satisfy $|\lambda|<\rho$.
  \end{itemize}
\end{theorem}
\begin{proof}
  Readers can refer to Chapter~7 of~\cite{Meyer2023}.
\end{proof}

\section{Bounds for the Spectral Radius of the $A_\alpha$-Matrix of Digraphs}\label{sec:3}
We start by giving an important lemma
which gives a simple lower bound of the spectral radius.

\begin{lemma}\label{lem:s-r-lower-bound-2}
  For a digraph $D$, we have
  \begin{equation}
    \rho(D)\ge\frac{\alpha m}{n}.
  \end{equation}
  Equality holds if and only if
  (i) $\alpha=0$ and $D$ is a DAG (directed acyclic graph);
  (ii) $0<\alpha<1$ and $D$ is an empty digraph; or
  (iii) $\alpha=1$ and $D$ is out-regular.
\end{lemma}
\begin{proof}[Proof of the inequality.]
  We have
  \begin{equation}
    \alpha m=M_1=|M_1|=\left|\sum_{i=1}^n\lambda_i\right|\le\sum_{i=1}^n|\lambda_i|\le\sum_{i=1}^n\rho(D)=n\rho(D)
    \implies
    \rho(D)\ge\frac{\alpha m}{n},
  \end{equation}
  completing the proof.
\end{proof}
\begin{proof}[Proof of the equality.]
  The proof is omitted for brevity
  since it is very similar to that of Theorem~\ref{thm:s-r-lower-bound-1},
  and the equality condition plays no role
  in the rest of the paper.
\end{proof}

We now derive a general upper bound
that depends only on the parameters $\alpha,\,n,\,m$ and $Z$.

\begin{theorem}\label{thm:s-r-upper-bound}
  For a digraph $D$ and $\alpha\in[0,1)$, we have
  \begin{equation}
    \rho(D)\le\frac{\alpha m}{n}+\sqrt{\frac{n-1}{n}\left(\alpha^2 Z+(1-\alpha)^2 m-\frac{(\alpha m)^2}{n}\right)}.
  \end{equation}
  The equality holds if and only if (i) $D\cong\overleftrightarrow{K_n}$;
  (ii) $D\cong\overleftrightarrow{K_k}\cup(n-k)K_1$ and $\alpha=1/k$;
  or (iii) $D$ is an empty graph.
\end{theorem}
\begin{proof}[Proof of the inequality.]
  By Lemma~\ref{lem:m1} and Theorem~\ref{thm:p-f-thm}, we have
  \begin{equation}\label{eq:s-r-upper-bound-eq1}
    \sum_{i=1}^n\lambda_i=\rho(D)+\sum_{i=2}^n\lambda_i=\alpha m
    \implies
    (\alpha m-\rho(D))^2=\left|\sum_{i=2}^n\lambda_i\right|^2.
  \end{equation}
  Applying Cauchy--Schwarz inequality, we obtain
  \begin{equation}\label{eq:s-r-upper-bound-eq2}
    \left|\sum_{i=2}^n\lambda_i\right|^2
    \le\left(\sum_{i=2}^n|\lambda_i|\right)^2
    \le(n-1)\sum_{i=2}^n|\lambda_i|^2.
  \end{equation}
  By Schur's inequality, it holds that
  \begin{equation}\label{eq:s-r-upper-bound-eq3}
    \sum_{i=1}^n|\lambda_i|^2=\rho(D)^2+\sum_{i=2}^n|\lambda_i|^2\le\|A_\alpha\|_F^2
    \implies
    \sum_{i=2}^n|\lambda_i|^2\le\|A_\alpha\|_F^2-\rho(D)^2.
  \end{equation}
  Finally,
  \labelcref{eq:s-r-upper-bound-eq1,eq:s-r-upper-bound-eq2,eq:s-r-upper-bound-eq3}
  gives
  \begin{equation}\label{eq:s-r-upper-bound-quadratic}
    (\alpha m-\rho(D))^2\le(n-1)(\|A_\alpha\|_F^2-\rho(D)^2)
    \implies
    n\rho(D)^2-2\alpha m\rho(D)+[(\alpha m)^2-(n-1)\|A_\alpha\|_F^2]\le 0,
  \end{equation}
  yielding
  \begin{equation}\label{eq:s-r-upper-bound-quadratic-2}
    \rho(D)\le\frac{2\alpha m+\sqrt{4(\alpha m)^2-4n[(\alpha m)^2-(n-1)\|A_\alpha\|_F^2]}}{2n}
    =\frac{\alpha m}{n}+\sqrt{\frac{n-1}{n}\left(\|A_\alpha\|_F^2-\frac{(\alpha m)^2}{n}\right)},
  \end{equation}
  which is the desired result.
\end{proof}
\begin{proof}[Proof of the equality condition.]
  Equality holds in \eqref{eq:s-r-upper-bound-eq3}
  if and only if $A_\alpha$ is a normal matrix.
  The second equality holds in \eqref{eq:s-r-upper-bound-eq2}
  if and only if $|\lambda_2|=|\lambda_3|=\dots=|\lambda_n|$.
  The first equality holds in \eqref{eq:s-r-upper-bound-eq2}
  if and only if $\arg\lambda_2=\arg\lambda_3=\dots=\arg\lambda_n$.
  Moreover, when the equality holds, since $\rho(D)\ge\frac{\alpha m}{n}$
  by Lemma~\ref{lem:s-r-lower-bound-2},
  $\rho(D)$ equals to the larger root
  of the quadratic equation in~\eqref{eq:s-r-upper-bound-quadratic}.
  To satisfy these conditions,
  $A_\alpha$ must be a normal matrix with eigenvalues $\rho(D)$
  and $\lambda_2=\lambda_3=\dots=\lambda_n\eqqcolon\lambda$.

  Therefore, if $\rho(D)\ne\lambda$, then
  by Lemma~\ref{lem:spectrum-complete-digraph-gen},
  either (i) $D\cong\overleftrightarrow{K_n}$;
  or (ii) $D\cong\overleftrightarrow{K_k}\cup(n-k)K_1$
  and $\alpha=1/k$.

  If $\rho(D)=\lambda$, then by Lemma~\ref{lem:s-r-lower-bound-2},
  we have
  \begin{equation}
    \alpha m=M_1=\sum_{i=1}^n \lambda_i=n\rho(D)\ge\alpha m
    \implies
    \rho(D)=\frac{\alpha m}{n}.
  \end{equation}
  By~\eqref{eq:s-r-upper-bound-quadratic-2}, this gives
  \begin{equation}
    \|A_\alpha\|_F^2-\frac{(\alpha m)^2}{n}
    =\alpha^2 Z+(1-\alpha)^2 m-\frac{\alpha^2 m^2}{n}=0.
  \end{equation}
  However, notice that
  \begin{equation}
    Z=\sum_v (d^+(v))^2\ge\frac{(\sum_v d^+(v))^2}{n}=\frac{m^2}{n},
  \end{equation}
  therefore
  \begin{equation}
    \frac{\alpha^2 m^2}{n}+(1-\alpha)^2 m-\frac{\alpha^2 m^2}{n}
    =(1-\alpha)^2 m\le 0
    \implies
    (1-\alpha)^2 m=0.
  \end{equation}
  By $1-\alpha\ne 0$, this forces $m=0$ and $D$ is an empty graph.
  This completes the proof.
\end{proof}

This upper bound requires no structural condition on the digraph.
An inspection suggests that this bound should perform
well for digraphs whose degrees are concentrated on
a small number of vertices.
We now present an numerical experiment
on a family of digraphs with this feature.

We propose \emph{core-complete digraphs} as follows.
First, choose a small vertex subset $V_\text{c}$ of fixed size $r$
and make its induced subdigraph
the complete symmetric digraph $\overleftrightarrow{K_r}$.
Call every vertex in $V_\text{c}$ a \emph{core vertex} and
every other vertex a \emph{secondary vertex}.
Then, for each secondary vertex $u$,
choose one core vertex $c_j$ with probability
\begin{equation}
  \Pr(\text{choose }c_j)=\frac{(1-\beta)\beta^{j-1}}{1-\beta^r},
  \qquad
  0<\beta<1,\,j=1,2,\dots,r.
\end{equation}
and form a digon with it.
This is the \emph{truncated geometric distribution}
and it promises that smaller values of $\beta$ make the attachments
more likely to concentrate on fewer core vertices.
Finally, add several additional random arcs
between $V_\text{c}$ and $V\setminus V_\text{c}$,
and within $V\setminus V_\text{c}$.
One can verify that such a digraph is strongly connected.
An illustration of this model is shown in
Figure~\ref{fig:core-complete-schematic}.

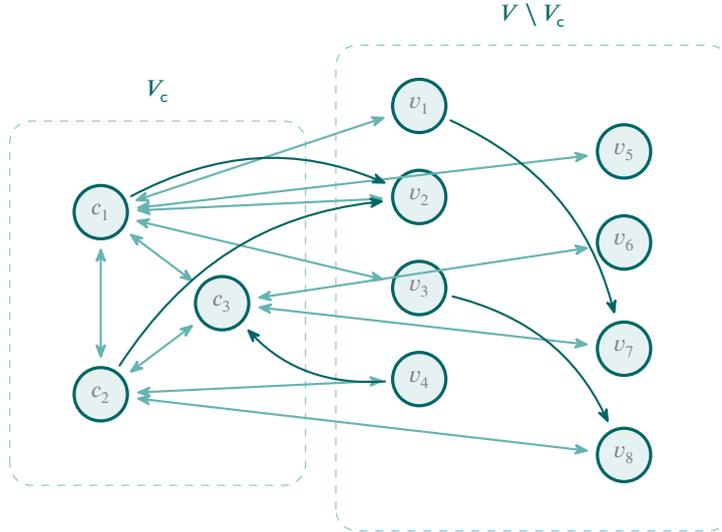
\begin{figure}[htbp]
  \centering
  \begin{tikzpicture}[
    >={Stealth[round, sep]},
    node distance=1.5cm,
    vertex/.style={
        circle,
        draw=teal!80!black,
        fill=teal!10,
        very thick,
        minimum size=20pt,
        inner sep=0pt,
        font=\small\bfseries\sffamily
    },
    digon/.style={
        <->,
        draw=teal!60,
        thick,
        shorten <=2pt,
        shorten >=2pt
    },
    inter/.style={
        ->,
        draw=teal!80!black,
        thick,
        bend left=25,
        shorten <=2pt,
        shorten >=2pt
    }
  ]

    \draw[rounded corners=8pt, dashed, draw=teal!50]
      (-1.2,-2.4) rectangle (2.7,2.4);
    \draw[rounded corners=8pt, dashed, draw=teal!50]
      (3.1,-3.0) rectangle (8.3,3.4);

    \node[font=\small\sffamily, text=teal!80!black] at (0.75,2.8) {$V_\text{c}$};
    \node[font=\small\sffamily, text=teal!80!black] at (5.7,3.8) {$V\setminus V_\text{c}$};

    \node[vertex] (c1) at (0,1.2) {$c_1$};
    \node[vertex] (c2) at (0,-1.2) {$c_2$};
    \node[vertex] (c3) at (1.6,0) {$c_3$};

    \node[vertex] (p1) at (4.2,2.6) {$v_1$};
    \node[vertex] (p2) at (4.2,1.4) {$v_2$};
    \node[vertex] (p3) at (4.2,0.2) {$v_3$};
    \node[vertex] (p4) at (4.2,-1.0) {$v_4$};
    \node[vertex] (p5) at (6.9,2.0) {$v_5$};
    \node[vertex] (p6) at (6.9,0.8) {$v_6$};
    \node[vertex] (p7) at (6.9,-0.6) {$v_7$};
    \node[vertex] (p8) at (6.9,-2.0) {$v_8$};

    \draw[digon] (c1) -- (c2);
    \draw[digon] (c1) -- (c3);
    \draw[digon] (c2) -- (c3);

    \draw[digon] (c1) -- (p1);
    \draw[digon] (c1) -- (p2);
    \draw[digon] (c1) -- (p3);
    \draw[digon] (c1) -- (p5);

    \draw[digon] (c2) -- (p4);
    \draw[digon] (c2) -- (p8);
    \draw[digon] (c3) -- (p6);
    \draw[digon] (c3) -- (p7);

    \draw[inter] (c2) to (p2);
    \draw[inter] (p4) to (c3);
    \draw[inter] (p3) to (p8);
    \draw[inter] (p1) to (p7);
    \draw[inter] (c1) to (p2);

  \end{tikzpicture}
  \caption{A schematic illustration of a core-complete digraph.
  For clarity, each digon is represented by a bidirectional arrow.}
  \label{fig:core-complete-schematic}
\end{figure}

Table~\ref{tbl:nc-1} presents a numerical comparison
of three upper bounds for the $A_\alpha$-matrix spectral radius
on random core-complete digraphs.
For each value of $\beta$,
we generate a large sample ($N=1,000$) of random digraphs
with $n=100,r=5$,
compute the exact spectral radius $\rho(D)$,
and evaluate the relative error $\frac{U}{\rho(D)}-1$,
where $U$ is one of the following bounds:
(i) the bound in Theorem~\ref{thm:s-r-upper-bound} (denoted by B1);
(ii) the bound in Theorem~2.2 of~\cite{Xi2021} (denoted by B2);
and (iii) the bound in Theorem~2.3 of~\cite{Xi2021} (denoted by B3).
Each entry in the table is reported as
$(\text{mean}\pm\text{standard deviation})$
of this relative error.

\begin{table}[htbp]
  \begin{tabular}{lccccc}
    \toprule
     & $\beta=0.80$ & $\beta=0.60$ & $\beta=0.40$ & $\beta=0.20$ & $\beta=0.10$ \\
    \midrule
    B1 & $0.9986\pm 0.0523$ & $0.8162\pm 0.0861$ & $0.5746\pm 0.0724$ & $0.3810\pm 0.0374$ & $0.3071\pm 0.0218$ \\
    B2 & $1.4756\pm 0.1353$ & $1.7752\pm 0.1033$ & $1.9556\pm 0.0478$ & $2.0471\pm 0.0240$ & $2.0771\pm 0.0145$ \\
    B3 & $1.0582\pm 0.0476$ & $1.0241\pm 0.0427$ & $0.9059\pm 0.0437$ & $0.7814\pm 0.0277$ & $0.7257\pm 0.0184$ \\
    \bottomrule
  \end{tabular}
  \begin{tabular}{lccccc}
    \toprule
     & $\beta=0.80$ & $\beta=0.60$ & $\beta=0.40$ & $\beta=0.20$ & $\beta=0.10$ \\
    \midrule
    B1 & $0.6738\pm 0.1208$ & $0.4161\pm 0.1023$ & $0.2234\pm 0.0558$ & $0.1132\pm 0.0224$ & $0.0782\pm 0.0121$ \\
    B2 & $0.4154\pm 0.0164$ & $0.4214\pm 0.0126$ & $0.4233\pm 0.0085$ & $0.4250\pm 0.0048$ & $0.4253\pm 0.0030$ \\
    B3 & $0.3142\pm 0.0249$ & $0.2573\pm 0.0226$ & $0.2069\pm 0.0140$ & $0.1713\pm 0.0073$ & $0.1570\pm 0.0046$ \\
    \bottomrule
  \end{tabular}
  \caption{Numerical comparison of three spectral radius bounds.
  Top: $\alpha=0.3$; Bottom: $\alpha=0.7$.
  Smaller is better.}\label{tbl:nc-1}
\end{table}

The results show that B1 becomes sharper as $\beta$ decreases.
For both $\alpha$,
the relative error of B1 drops significantly when the attachments
become more concentrated on fewer core vertices.
The bound B3 performs much better than B2, but is still weaker
than B1 when $\beta$ is small.
This supports our prediction
that B1 is especially effective for core-complete digraphs.

\begin{corollary}
  When $\alpha=0$, this reduces to a classical upper bound
  for the spectral radius of the adjacency matrix
  of a digraph,
  \begin{equation}
    \rho(A(D))\le\sqrt{\frac{n-1}{n}m}.
  \end{equation}
\end{corollary}

We also derive a lower bound for the spectral radius.

\begin{theorem}\label{thm:s-r-lower-bound-1}
  Let $D$ be a digraph of order $n$ and let $\alpha\in[0,1]$.
  Then
  \begin{equation}\label{eq:sr-lb1}
    \rho(D)\ge\sqrt{\frac{\alpha^2 Z+(1-\alpha)^2 c_2}{n}}.
  \end{equation}
  The equality holds if and only if one of the following occurs:
  \begin{itemize}
    \item[(i)] $\alpha=0$ and either
    1) $D$ is a DAG (directed acyclic graph);
    or 2) every strong component of $D$ is a digon $\overleftrightarrow{K_2}$;
    \item[(ii)] $0<\alpha<1$ and $D$ is an empty digraph;
    \item[(iii)] $\alpha=1$ and $D$ is out-regular,
    i.e. $d^+(v)$ is constant for all $v\in V$.
  \end{itemize}
\end{theorem}
\begin{proof}[Proof of the inequality.]
  We have
  \begin{equation}
    \alpha^2 Z+(1-\alpha)^2 c_2=M_2=|M_2|
    =\left|\sum_{i=1}^n\lambda_i^2\right|
    \le\sum_{i=1}^n|\lambda_i|^2
    \le\sum_{i=1}^n\rho(D)^2
    =n\rho(D)^2,
  \end{equation}
  which implies~\eqref{eq:sr-lb1}.
\end{proof}

\begin{proof}[Proof of the equality condition.]
  Write $\rho=\rho(D)$.
  The second equality holds if and only if
  $|\lambda_1|=\dots=|\lambda_n|=\rho$.
  The first equality holds if and only if
  all $\arg\lambda_1^2=\arg\lambda_2^2=\dots=\arg\lambda_n^2$.
  However, since
  $\sum_{i=1}^n\lambda_i^2=M_2=\alpha^2 Z+(1-\alpha)^2c_2\in\R_{\ge 0}$,
  it follows that each $\lambda_i^2\in\R_{\ge 0}$.
  Therefore $\lambda_i\in\{\rho,-\rho\}$ and thus
  $\spec(A_\alpha(D))\subseteq\{\rho,-\rho\}$.

  \vspace{0.5em}\noindent\textbf{Case $\alpha=0$.}
  Then $A_\alpha=A(D)$.
  If $\rho=0$, then all eigenvalues of $A(D)$ are $0$, so $A(D)$ is nilpotent.
  For digraph $D$, $A(D)$ is nilpotent if and only if
  $D$ has no directed cycle, i.e., $D$ is a DAG
  (see~\cite{Brauldi1991}, Chapter~9.8).
  If $\rho>0$, then by Lemma~\ref{lem:spectrum-k2},
  it must hold that $\rho=1$ and every strong component of $D$ is a digon
  $\overleftrightarrow{K_2}$.

  \vspace{0.5em}\noindent\textbf{Case $0<\alpha<1$.}
  If $\rho=0$, then all eigenvalues of $A_\alpha(D)$ are $0$,
  so $\tr(A_\alpha(D))=\alpha m=0$.
  Since $\alpha>0$, it must holds that $m=0$,
  i.e., $D$ is an empty digraph.

  Now assume $\rho>0$.
  Let $S$ be a sink strong component of $D$,
  and let $B=A_\alpha(S)$ be the principal submatrix of $A_\alpha(D)$
  indexed by $V(S)$.
  Since $S$ is a strong component, $B$ occurs as
  a diagonal block in the rational normal form of $A_\alpha(D)$.
  This gives $\spec(B)\subseteq\spec(A_\alpha(D))\subseteq\{\rho,-\rho\}$.

  If $|V(S)|=1$, say $V(S)=\{v\}$, then
  $d^+(v)=0,B=[0]$ and thus $0\in\spec(A_\alpha(D))$,
  contradicting $\spec(A_\alpha(D))\subseteq\{\rho,-\rho\}$ with $\rho>0$.

  If $|V(S)|\ge 2$, since $S$ is strongly connected, every vertex in $S$
  has out-degree at least $1$,
  and hence every diagonal entry of $B$ is positive.
  Therefore $B$ is a nonnegative irreducible matrix with a
  positive diagonal entry, and thus $B$ is primitive.
  By Theorem~\ref{thm:p-f-thm}, $B$ has exactly one eigenvalue
  on the circle $|\mu|=\rho(B)$,
  and all other eigenvalues satisfy $|\mu|<\rho(B)$.
  Since $B$ has order at least $2$, it has an eigenvalue
  different from $\rho(B)$,
  forcing some eigenvalue $\mu$ exists with $|\mu|<\rho(B)$.
  This contradicts $\spec(B)\subseteq\{\rho,-\rho\}$.

  Hence $\rho>0$ is impossible, forcing $\rho=0$.
  As shown above,
  this implies $m=0$, i.e., $D$ is empty.

  \vspace{0.5em}\noindent\textbf{Case $\alpha=1$.}
  Then $A_\alpha(D)=\Deg(D)$ is diagonal and its eigenvalues are
  $\{d^+(v):v\in V\}$. Hence
  \begin{equation}
    \rho(D)=\max_{v\in V} d^+(v),\qquad
    \alpha^2 Z+(1-\alpha)^2c_2 = Z = \sum_{v\in V} (d^+(v))^2.
  \end{equation}
  Equality in~\eqref{eq:sr-lb1} becomes equality in
  $\max_{v\in V}d^+(v)\ge \sqrt{\frac{1}{n}\sum (d^+(v))^2}$, which holds if and only if
  $d^+(v)$ is constant for all $v$, i.e., $D$ is out-regular.
  This completes the proof.
\end{proof}

\section{Koolen-Moulton Type Bounds for the Low Energy of the $A_\alpha$-matrix}\label{sec:4}
Low energy is a natural extension of graph energy
to digraphs that preserves Coulson's integral formula~\cite{PenaRada2008}.
Independently, in the setting of undirected graphs
and Gutman's original definition of graph energy,
Koolen and Moulton~\cite{Koolen2001} isolated the largest eigenvalue
and then bounded the contribution of the remaining parts,
yielding a sharper upper bound for the energy.

In this subsection, following the idea of Koolen and Moulton,
we derive two upper bounds for the low energy,
one involving $\rho(D)$ and one independent of $\rho(D)$.

\begin{theorem}\label{thm:k-m-bound}
  For a digraph $D$, we have
  \begin{equation}
    E(D)\le\rho(D)+\sqrt{(n-1)\left(\alpha^2 Z+(1-\alpha)^2\frac{m+c_2}{2}-\rho(D)^2\right)}.
  \end{equation}
  Equality holds if and only if $|x_2|=|x_3|=\dots=|x_n|$
  and $A_\alpha$ is normal.
\end{theorem}
\begin{proof}[Proof of the inequality.]
  By Lemma~\ref{lem:m2}, we have
  \begin{equation}
    \sum_{i=1}^n\lambda_i^2
    =\sum_{i=1}^n (x_i+\iu y_i)^2
    =\left(\sum_{i=1}^n x_i^2-\sum_{i=1}^n y_i^2\right)+2\iu\sum_{i=1}^n x_iy_i
    =\alpha^2 Z+(1-\alpha)^2 c_2+0\iu,
  \end{equation}
  which gives
  \begin{equation}\label{eq:square-sum}
    \sum_{i=1}^n x_iy_i=0
    \qquad\text{and}\qquad
    \sum_{i=1}^n x_i^2-\sum_{i=1}^n y_i^2=\alpha^2 Z+(1-\alpha)^2 c_2.
  \end{equation}
  
  On the other hand, by Schur's Inequality, we have
  \begin{equation}\label{eq:square-diff}
    \sum_{i=1}^n x_i^2+\sum_{i=1}^n y_i^2=\sum_{i=1}^n|\lambda_i|^2
    \le\|A_\alpha\|_F^2
    =\alpha^2 Z+(1-\alpha)^2 m.
  \end{equation}
  Combining~\labelcref{eq:square-sum,eq:square-diff}, we have
  \begin{equation}\label{eq:T-bound}
    \sum_{i=1}^n x_i^2\le\alpha^2 Z+(1-\alpha)^2\frac{m+c_2}{2}.
  \end{equation}

  By Perron--Frobenius Theorem,
  $\rho(D)$ is a real eigenvalue of $A_\alpha$.
  Without loss of generality, write $x_1=\rho(D)$.
  Finally, by Cauchy--Schwarz Inequality, we have
  \begin{equation}\label{eq:k-m-bound-eq2}
    \sum_{i=2}^n|x_i|=\sqrt{\left(\sum_{i=2}^n|x_i|\right)^2}
    \le\sqrt{(n-1)\sum_{i=2}^n x_i^2}
    \le\sqrt{(n-1)\left(\alpha^2 Z+(1-\alpha)^2\frac{m+c_2}{2}-\rho(D)^2\right)}.
  \end{equation}
  Therefore
  \begin{equation}
    E(A_\alpha)=\sum_{i=1}^n|x_i|=\rho(D)+\sum_{i=2}^n|x_i|
    \le\rho+\sqrt{(n-1)\left(\alpha^2 Z+(1-\alpha)^2\frac{m+c_2}{2}-\rho(D)^2\right)},
  \end{equation}
  completing the proof.
\end{proof}
\begin{proof}[Proof of the equality.]
  Equality in the first inequality in~\eqref{eq:k-m-bound-eq2}
  holds if and only if the equality in Cauchy--Schwarz inequality holds, i.e.,
  \begin{equation}
    \left(\sum_{i=2}^n |x_i|\right)^2=(n-1)\sum_{i=2}^n x_i^2,
  \end{equation}
  which holds if and only if $|x_2|=|x_3|=\dots=|x_n|$.
  Equality in the second inequality in~\eqref{eq:k-m-bound-eq2}
  holds if and only if the equality in Schur's inequality holds, i.e.,
  \begin{equation}
    \sum_{i=1}^n|\lambda_i|^2=\|A_\alpha\|_F^2,
  \end{equation}
  which is equivalent to $A_\alpha$ being a normal matrix.
  This completes the proof.
\end{proof}

\begin{corollary}\label{cor:km-bound-no-rho}
  By Theorem~\ref{thm:k-m-bound},
  we can further obtain an upper bound for $n>1$
  that is irrelevant to $\rho(D)$,
  \begin{equation}
    E(D)\le\sqrt{n\left(\alpha^2 Z+(1-\alpha)^2\frac{m+c_2}{2}\right)}.
  \end{equation}
  Equality holds if and only if $A_\alpha(D)$ is normal and
  $|x_1|=|x_2|=\cdots=|x_n|$. In particular:
  (i) if $\alpha=0$, equality holds if and only if $D$ is empty or
  $D\cong\frac{n}2\overleftrightarrow{K_2}$;
  (ii) if $0<\alpha<1$, equality holds if and only if
  $D$ is empty;
  (iii) if $\alpha=1$, equality holds if and only if
  $D$ is out-regular.
\end{corollary}
\begin{proof}[Proof of the inequality]
  For simplicity, write
  \begin{equation}
    T\coloneqq \alpha^2 Z+(1-\alpha)^2\frac{m+c_2}{2}.
  \end{equation}
  By Theorem~\ref{thm:k-m-bound}, we have
  \begin{equation}
    E(D)\le f(\rho(D)),
    \qquad
    f(t)\coloneqq t+\sqrt{(n-1)\left(T-t^2\right)}.
  \end{equation}
  Notice that~\eqref{eq:T-bound} implies
  $\rho(D)^2\le T\implies\rho(D)\in[0,\sqrt{T}]$
  and $f$ is well-defined on this interval.

  For $t\in(0,\sqrt{T})$, we have
  \begin{equation}
    f'(t)=1-\frac{\sqrt{n-1}\,t}{\sqrt{T-t^2}},
    \qquad
    f''(t)=-\frac{T\sqrt{n-1}}{(T-t^2)^{3/2}}<0,
  \end{equation}
  and $f'(t)=0$ if and only if
  $\sqrt{n-1}\,t=\sqrt{T-t^2}\Longleftrightarrow t=\sqrt{\frac{T}{n}}$.
  Therefore $f$ attains its maximum on $[0,\sqrt{T}]$ at
  $t_0=\sqrt{\frac{T}{n}}$,
  and hence $E(D)\le f(\rho(D))\le f(t_0)$.
  Finally,
  \begin{equation}
    f(t_0)=\sqrt{\frac{T}{n}}+\sqrt{(n-1)\left(T-\frac{T}{n}\right)}
    =\sqrt{\frac{T}{n}}+(n-1)\sqrt{\frac{T}{n}}
    =\sqrt{nT},
  \end{equation}
  completing the proof.
\end{proof}
\begin{proof}[Proof of the equality]
  Both inequalities in
  \begin{equation}
    E(D)\le f(\rho(D))\le f(t_0)
  \end{equation}
  must be equalities for the equality to hold.

  Equality in the first inequality
  is exactly the equality case of Theorem~\ref{thm:k-m-bound}.
  Hence $A_\alpha(D)$ is normal and $|x_2|=|x_3|=\dots=|x_n|$.

  Equality in the second inequality implies $\rho(D)=t_0=\sqrt{\frac{T}{n}}$.
  Therefore
  \begin{equation}
    \rho(D)^2=\frac{T}{n}.
  \end{equation}
  On the other hand, by~\eqref{eq:T-bound}, we have
  \begin{equation}
    \sum_{i=2}^n x_i^2=T-\rho(D)^2.
  \end{equation}
  Since $|x_2|=|x_3|\dots=|x_n|$, it follows that
  \begin{equation}
    |x_2|^2=|x_3|^2=\dots=|x_n|^2
    =\frac{T-\rho(D)^2}{n-1}
    =\frac{T-\frac{T}{n}}{n-1}
    =\frac{T}{n}
    =\rho(D)^2,
  \end{equation}
  which gives
  \begin{equation}
    \rho(D)=|x_1|=|x_2|=|x_3|=\dots=|x_n|.
  \end{equation}
  
  This corresponds to the extremal digraph described in
  Theorem~\ref{thm:s-r-lower-bound-1}
  with an additional requirement that $A_\alpha(D)$ is normal.
  This extra condition does change the required graph structure,
  because since $A_\alpha(D)$ is a real normal matrix
  with only real eigenvalues, it must be symmetric.
  Consequently, the digraph $D$ must also be symmetric.
  Therefore, in the nontrivial equality case for $\alpha=0$,
  there can no longer be arcs between distinct digons.
  In other words,
  the digraph satisfying the equality when $\alpha=0$
  must be a disjoint union of copies of
  $\overleftrightarrow{K_2}$.
\end{proof}

In the following two corollaries,
we demonstrate that the two bounds derived above
reduce to several classical bounds,
therefore they can be seen as a generalization
of these existing results.

\begin{corollary}\label{cor:km-bound-undirected}
  Consider a symmetric digraph $G$ where each pair of adjacent vertices
  are connected by a digon.
  Then $G$ is equivalent to an undirected graph
  with $n'=n$ vertices and $m'=m/2=c_2/2$ edges,
  and $\Deg(G),A(G)$ reduces to
  its degree matrix and adjacency matrix, respectively.
  Then, by Theorem~\ref{thm:k-m-bound}, we have
  \begin{equation}
    E(A_\alpha)\le\rho(G)+\sqrt{(n'-1)\left(\alpha^2 Z+2(1-\alpha)^2m'-\rho(G)^2\right)}.
  \end{equation}
  This is a Koolen-Moulton type bound for the $A_\alpha$-matrix
  of an undirected graph.

  Similarly, by Corollary~\ref{cor:km-bound-no-rho}, we have
  \begin{equation}
    E(A_\alpha)\le\sqrt{n'\left(\alpha^2 Z+2(1-\alpha)^2 m'\right)}.
  \end{equation}
\end{corollary}

\begin{corollary}
  By setting $\alpha=0$ in Corollary~\ref{cor:km-bound-undirected},
  we have
  \begin{equation}
    E(A)\le\rho(G)+\sqrt{(n'-1)\left(2m'-\rho(G)^2\right)}
  \end{equation}
  and
  \begin{equation}
    E(A)\le\sqrt{2n'm'},
  \end{equation}
  respectively.
  For the former bound, notice that function
  $f(t)=t+\sqrt{(n'-1)\left(2m'-t^2\right)}$
  reaches its maximal at $t=\sqrt{\frac{2m'}{n'}}$
  and decreases on $\left[\sqrt{\frac{2m'}{n'}},+\infty\right)$.
  Therefore, for any undirected graph $G$ such that
  \begin{equation}
    \frac{2m'}{n'}\ge\sqrt{\frac{2m'}{n'}}
    \Longleftrightarrow
    \frac{2m'}{n'}\ge 1,
  \end{equation}
  by the inequality $\rho(G)\ge\frac{2m'}{n'}$, we obtain
  \begin{equation}
    E(A(G))\le\frac{2m'}{n'}+\sqrt{(n'-1)\left(2m'-\frac{4m'^2}{n'^2}\right)}.
  \end{equation}
  which is exactly the classical Koolen-Moulton bound for the adjacency matrix
  of an undirected graph,
  given by Koolen and Moulton in Theorem~1 of~\cite{Koolen2001}.

  The latter bound is also a well-known classical bound
  for undirected graphs, first introduced by McClelland~\cite{McClelland1971}.
\end{corollary}

Table~\ref{tbl:nc-2} presents a numerical comparison
of three upper bounds for the $A_\alpha$-matrix low energy $E(D)$
on randomly generated $n$-vertex $k$-regular digraphs,
where $k=6,7,8,9,10$
and the corresponding values of $n$ are listed in the table.
For each $(n,k)$, we generate a large sample ($N=1,000$)
of random $k$-regular digraphs, compute $E(D)$ for each sample
and then evaluate the relative error $\frac{U}{E(D)}-1$,
where $U$ is one of the following bounds:
(i) the bound in Theorem~\ref{thm:k-m-bound} (denoted by B1);
(ii) the bound in Corollary~\ref{cor:km-bound-no-rho} (denoted by B2);
and (iii) the existing bound for regular digraphs
(Corollary~3.2 of~\cite{Pirzada2025}, denoted by B3).
Each table entry is $(\text{mean}\pm\text{standard deviation})$
of the above relative error.

\begin{table}[htbp]
  \begin{tabular}{lccccc}
    \toprule
     & $k=6$ & $k=7$ & $k=8$ & $k=9$ & $k=10$ \\
    & $(n=60)$ & $(n=70)$ & $(n=80)$ & $(n=90)$ & $(n=100)$ \\
    \midrule
    B1 & $0.1922\pm 0.0042$ & $0.1662\pm 0.0033$ & $0.1463\pm 0.0024$ & $0.1308\pm 0.0020$ & $0.1181\pm 0.0017$ \\
    B2 & $0.2243\pm 0.0040$ & $0.1950\pm 0.0031$ & $0.1725\pm 0.0023$ & $0.1546\pm 0.0019$ & $0.1400\pm 0.0016$ \\
    B3 & $0.7064\pm 0.0069$ & $0.6542\pm 0.0057$ & $0.6121\pm 0.0045$ & $0.5772\pm 0.0038$ & $0.5474\pm 0.0033$ \\
    \bottomrule
  \end{tabular}
  \begin{tabular}{lccccc}
    \toprule
     & $k=6$ & $k=7$ & $k=8$ & $k=9$ & $k=10$ \\
    & $(n=60)$ & $(n=70)$ & $(n=80)$ & $(n=90)$ & $(n=100)$ \\
    \midrule
    B1 & $0.0069 \pm 0.0002$ & $0.0059 \pm 0.0001$ & $0.0052 \pm 0.0001$ & $0.0046 \pm 0.0001$ & $0.0041 \pm 0.0001$ \\
    B2 & $0.0084 \pm 0.0002$ & $0.0072 \pm 0.0001$ & $0.0063 \pm 0.0001$ & $0.0056 \pm 0.0001$ & $0.0050 \pm 0.0001$ \\
    B3 & $0.1298 \pm 0.0013$ & $0.1203 \pm 0.0010$ & $0.1125 \pm 0.0009$ & $0.1060 \pm 0.0007$ & $0.1006 \pm 0.0006$ \\
    \bottomrule
  \end{tabular}
  \caption{Numerical comparison of three low energy bounds.
  Top: $\alpha=0.3$; Bottom: $\alpha=0.7$.
  Smaller is better.}\label{tbl:nc-2}
\end{table}

The results indicate that B1 that we derived
is consistently the sharpest bound
for all tested values of $(n,k)$ and for both choices of $\alpha$
shown in the two parts of the table.
The bound B2 is slightly weaker than B1,
but it is still noticeably sharper than B3.

\section{Extremal Graph Families for the Equality Case of the Low Energy Bound}\label{sec:5}
This section serves as a complement
for the equality conditions in the low energy upper bounds
established in Theorem~\ref{thm:k-m-bound}
and Corollary~\ref{cor:km-bound-no-rho}.
Recall that the equality occurs if and only if
$A_\alpha(D)$ is a normal matrix and $|x_2|=|x_3|=\cdots=|x_n|$.
In this section, we give several families of digraphs
for which these two spectral properties hold.

We start with converting the algebraic condition that
the $A_\alpha$-matrix is normal
to a topological condition on the digraph.

\begin{theorem}\label{thm:normal-cond}
  Let $D=(V,E)$ be a digraph of order $n$ and let $\alpha\in[0,1)$.
  For distinct vertices $u,v\in V$, define
  \begin{align}
    &\delta_{uv}\coloneqq |N^+(u)\cap N^+(v)|-|N^-(u)\cap N^-(v)|, \\
    &\Delta_{uv}\coloneqq d^+(u)-d^+(v),
    \qquad
    \sigma_{uv}\coloneqq A_{uv}-A_{vu}\in\{-1,0,1\}.
  \end{align}
  Then $A_\alpha$ is normal if and only if
  \begin{equation}\label{eq:normal-diag-balance}
    d^+(u)=d^-(u),\quad\forall u\in V
    \qquad\text{and}\qquad
    \delta_{uv}=\frac{\alpha}{1-\alpha}\Delta_{uv}\sigma_{uv},\quad\forall u,v\in V,u\ne v.
  \end{equation}
\end{theorem}
\begin{proof}
  Since $A_\alpha$ is real, $A_\alpha$ is normal if and only if
  $A_\alpha A_\alpha^\tp=A_\alpha^\tp A_\alpha$.
  For simplicity, write $\Delta=\Deg(D)$ and $\beta=1-\alpha$.
  Then $A_\alpha=\alpha\Delta+\beta A$, and hence
  \begin{align}\label{eq:normal-expand}
    A_\alpha A_\alpha^\tp-A_\alpha^\tp A_\alpha
    &=(\alpha\Delta+\beta A)(\alpha\Delta+\beta A^\tp)
    -(\alpha\Delta+\beta A^\tp)(\alpha\Delta+\beta A) \notag\\
    &=\beta^2(AA^\tp-A^\tp A)
    +\alpha\beta(\Delta A^\tp+A\Delta-\Delta A-A^\tp\Delta).
  \end{align}
  Therefore $A_\alpha$ is normal if and only if the matrix on the right-hand side
  of~\eqref{eq:normal-expand} is zero.

  First, we consider diagonal entries.
  For any $u\in V$, we have
  \begin{equation}
    (AA^\tp)_{uu}=d^+(u),
    \qquad
    (A^\tp A)_{uu}=d^-(u)
  \end{equation}
  and
  \begin{equation}
    A_{uu}=0\implies
    (\Delta A^\tp)_{uu}=(A\Delta)_{uu}=(\Delta A)_{uu}=(A^\tp\Delta)_{uu}=0.
  \end{equation}
  Thus the $(u,u)$-entry of~\eqref{eq:normal-expand} equals
  $(1-\alpha)^2\bigl(d^+(u)-d^-(u))$.
  Since $\alpha\ne 1$, this vanishes for every $u$ if and only if
  \begin{equation}
    d^+(u)=d^-(u),\quad\forall u\in V.
  \end{equation}

  Next consider off-diagonal entries $u\ne v$. We have
  \begin{equation}
    (AA^\tp)_{uv}=\sum_{w\in V}A_{uw}A_{vw},
    \qquad
    (A^\tp A)_{uv}=\sum_{w\in V}A_{wu}A_{wv},
  \end{equation}
  where the first sum counts the common out-neighbors of $u$ and $v$,
  while the second counts their common in-neighbors.
  Hence
  \begin{equation}
    (AA^\tp)_{uv}=|N^+(u)\cap N^+(v)|,
    \qquad
    (A^\tp A)_{uv}=|N^-(u)\cap N^-(v)|.
  \end{equation}
  Also, since $\Delta$ is diagonal, we have
  \begin{equation}
    (\Delta A^\tp)_{uv}=d^+(u)A_{vu},\quad
    (A\Delta)_{uv}=A_{uv}d^+(v),\quad
    (\Delta A)_{uv}=d^+(u)A_{uv},\quad
    (A^\tp\Delta)_{uv}=A_{vu}d^+(v).
  \end{equation}
  Hence
  \begin{equation}
    (\Delta A^\tp+A\Delta-\Delta A-A^\tp\Delta)_{uv}
    =d^+(u)A_{vu}+A_{uv}d^+(v)-d^+(u)A_{uv}-A_{vu}d^+(v)
    =(d^+(u)-d^+(v))(A_{vu}-A_{uv}).
  \end{equation}
  Substituting this into~\eqref{eq:normal-expand},
  the $(u,v)$-entry becomes
  $(1-\alpha)^2\delta_{uv}-\alpha(1-\alpha)\Delta_{uv}\sigma_{uv}$.
  This vanishes if and only if
  \begin{equation}
    \delta_{uv}=\frac{\alpha}{1-\alpha}\Delta_{uv}\sigma_{uv},\quad\forall u,v\in V,u\ne v.
  \end{equation}
  This completes the proof.
\end{proof}

\begin{corollary}
  The condition
  \begin{equation}
    \delta_{uv}=\frac{\alpha}{1-\alpha}\Delta_{uv}\sigma_{uv},\quad\forall u,v\in V,u\ne v
  \end{equation}
  in Theorem~\ref{thm:normal-cond} can be further simplified.
  If $\frac{\alpha}{1-\alpha}=\frac{p}{q}\in\mathbb{Q},\,p,q\in\mathbb{Z}_{\ge 0},\,\gcd(p,q)=1,$
  then it simplifies to
  \begin{equation}
    \begin{cases}
      \exists k_{uv}\in\mathbb{Z}:\,q\sigma_{uv}k_{uv}=\Delta_{uv},\,pk_{uv}=\delta_{uv},
      &\forall u,v\in V,u\ne v,\sigma_{uv}\ne 0; \\
      \delta_{uv}=0, &\forall u,v\in V,u\ne v,\sigma_{uv}=0.
    \end{cases}
  \end{equation}
  If $\alpha\notin\mathbb{Q}$, then it simplifies to
  \begin{equation}
    \delta_{uv}=0,\,\Delta_{uv}\sigma_{uv}=0,\qquad\forall u,v\in V,u\ne v.
  \end{equation}
\end{corollary}
\begin{proof}
  The proof is straightforward.
\end{proof}

We demonstrate that the regular tournament
is the unique graph among all tournament graphs
that satisfies the condition.

\begin{theorem}
  Let $D$ be a tournament of order $n$ and let $\alpha\in[0,1)$.
  Then $A_\alpha(D)$ is a normal matrix satisfying
  $|x_2|=|x_3|=\dots=|x_n|$ if and only if $D$ is a regular tournament.
\end{theorem}
\begin{proof}[Proof of Necessity ($\implies$)]
  Suppose $A_\alpha(D)$ is normal.
  By Theorem~\ref{thm:normal-cond}, normality implies
  $d^+(u)=d^-(u),\,\forall u\in V$,
  which is exactly the definition of $D$ being a regular tournament.
\end{proof}
\begin{proof}[Proof of Sufficiency ($\impliedby$)]
  Assume that $D$ is a regular tournament.
  We first show that $A_\alpha(D)$ is normal.
  Put $k=(n-1)/2$, so $d^+(u)=d^-(u)=k$ for all $u\in V$.
  Hence
  \begin{equation}
    d^+(u)=d^-(u),\quad \forall u\in V,
    \qquad
    \Delta_{uv}=0,\quad \forall u\ne v.
  \end{equation}
  Now fix $u\ne v$. Since $D$ is a tournament,
  exactly one of $u\to v$ and $v\to u$ holds.
  Without loss of generality, assume $u\to v$.
  Then since $D$ is a tournament,
  every vertex $w\in V\setminus\{u,v\}$
  belongs to exactly one of the four sets
  \begin{equation}
    N^+(u)\cap N^+(v),\quad
    N^+(u)\cap N^-(v),\quad
    N^-(u)\cap N^+(v),\quad
    N^-(u)\cap N^-(v).
  \end{equation}
  Moreover,
  \begin{equation}
    N^+(u)=\{v\}\cup\bigl(N^+(u)\cap N^+(v)\bigr)\cup\bigl(N^+(u)\cap N^-(v)\bigr),
  \end{equation}
  so
  \begin{equation}
    d^+(u)=1+|N^+(u)\cap N^+(v)|+|N^+(u)\cap N^-(v)|.
  \end{equation}
  Similarly,
  \begin{equation}
    N^-(v)=\{u\}\cup\bigl(N^+(u)\cap N^-(v)\bigr)\cup\bigl(N^-(u)\cap N^-(v)\bigr),
  \end{equation}
  and hence
  \begin{equation}
    d^-(v)=1+|N^+(u)\cap N^-(v)|+|N^-(u)\cap N^-(v)|.
  \end{equation}
  Since $d^+(u)=d^-(v)$, it follows that
  \begin{equation}
    |N^+(u)\cap N^+(v)|=|N^-(u)\cap N^-(v)|.
  \end{equation}
  By Theorem~\ref{thm:normal-cond}, $A_\alpha(D)$ is normal.

  Next we prove that $|x_2|=|x_3|=\dots=|x_n|$.
  Let $A=A(D)=A_\alpha(D)|_{\alpha=0}$,
  which is normal as we just proved.
  Since $D$ is a regular tournament, it holds that
  \begin{equation}
    A\mathbf{1}=k\mathbf{1},\qquad A+A^\tp=J-I,
  \end{equation}
  where $\mathbf{1}$ is the all-$1$ vector and
  $J$ is the all-$1$ matrix.

  Let $\lambda$ be any eigenvalue of $A$ other than $k$, and let
  $\mathbf{x}$ be a corresponding eigenvector.
  Since eigenspaces corresponding to distinct eigenvalues
  of a normal matrix are orthogonal, we have
  $\mathbf{x}\perp\mathbf{1}$.
  Hence $J\mathbf{x}=(\mathbf{1}\mathbf{1}^\tp)\mathbf{x}=\mathbf{1}(\mathbf{1}^\tp\mathbf{x})=\mathbf{0}$
  and
  \begin{equation}
    (A+A^\tp)\mathbf{x}=(J-I)\mathbf{x}=\mathbf{0}-\mathbf{x}=-\mathbf{x}.
  \end{equation}
  Moreover, since $A$ is normal,
  we have
  $A\mathbf{x}=\lambda\mathbf{x}\implies A^\tp\mathbf{x}=\overline{\lambda}\mathbf{x}$.
  Therefore
  \begin{equation}
    (A+A^\tp)\mathbf{x}
    =(\lambda+\overline{\lambda})\mathbf{x}
    =-\mathbf{x}
    \quad\implies\quad
    \rp(\lambda)=-\frac{1}{2},
  \end{equation}
  i.e., every non-Perron eigenvalue of $A$ has real part $-1/2$.
  Therefore the real parts of the non-Perron eigenvalues
  of $A_\alpha=\alpha kI+(1-\alpha)A$ are
  \begin{equation}
    \alpha k\cdot 1+(1-\alpha)\cdot\left(-\frac{1}{2}\right)
    =\frac{\alpha n-1}{2}.
  \end{equation}
  Hence
  \begin{equation}
    |x_2|=|x_3|=\cdots=|x_n|=\left|\frac{\alpha n-1}{2}\right|,
  \end{equation}
  completing the proof.
\end{proof}

For directed cycles, only specific configurations
satisfy the condition, and there are restrictions on $\alpha$.

\begin{theorem}
  Let $D\cong\overrightarrow{C_k}$ be a directed $k$-cycle and let $\alpha\in[0,1)$.
  Then $A_\alpha(D)$ is normal. Moreover,
  $|x_2|=|x_3|=\dots=|x_k|$ holds if and only if one of the following occurs:
  (i) $k=3$;
  (ii) $k=4$ and $\alpha=1/3$;
  (iii) $k=5$ and $\alpha=1/5$.
\end{theorem}
\begin{proof}
  Notice that the adjacency matrix of $\overrightarrow{C_k}$
  is the $k$-cycle permutation matrix $P$,
  so $PP^\tp=I$ and $P$ is normal.
  Hence $A_\alpha=\alpha I+(1-\alpha)P$ is also normal.
  The eigenvalues of $P$ are $\omega^j$ where
  $\omega=\exp(2\pi\iu/k)$,
  so the eigenvalues of $A_\alpha$ are
  $\lambda_j=\alpha+(1-\alpha)\omega^j$ and
  their real parts are
  $x_j=\alpha+(1-\alpha)\cos(2\pi j/k)$.

  The condition $|x_2|=|x_3|=\dots=|x_k|$ is equivalent to that
  $|\alpha+(1-\alpha)\cos(2\pi j/k)|$
  is constant for all $j=1,2,\dots,k-1$.
  For $k\ge 6$, this gives at least three distinct values, a contradiction.
  Hence $k\in\{3,4,5\}$.

  If $k=3$, then $\cos(2\pi/3)=\cos(4\pi/3)=-1/2$, so
  $x_2=x_3=\frac{3\alpha-1}{2}$
  and the condition holds for all $\alpha$.

  If $k=4$, then $\cos(2\pi/4)=\cos(6\pi/4)=0$ and $\cos(4\pi/4)=-1$,
  so $x_2=\alpha,x_3=2\alpha-1,x_4=\alpha$
  and the condition is $|\alpha|=|2\alpha-1|$.
  With $\alpha\in[0,1)$, this gives $\alpha=1/3$.

  If $k=5$, write $c_1=\cos(2\pi/5)$ and $c_2=\cos(4\pi/5)$. Since $c_1\ne c_2$,
  the condition forces
  $\alpha+(1-\alpha)c_1=-[\alpha+(1-\alpha)c_2]$.
  By $c_1+c_2=-1/2$, this gives $\alpha=1/5$.
\end{proof}

Finally, we establish that any complete bipartite graph
satisfies the condition if and only if $\alpha=1/3$.

\begin{theorem}
  Let $D\cong\overleftrightarrow{K_{t,t}},\,t>1$ and let $\alpha\in[0,1)$.
  Then $A_\alpha(D)$ is normal and satisfies $|x_2|=|x_3|=\dots=|x_{2t}|$
  if and only if $\alpha=1/3$.
\end{theorem}
\begin{proof}
  Since $D$ is symmetric, $A(D)$ and $A_\alpha(D)$ are also symmetric and thus is normal.
  Moreover, the underlying undirected graph $K_{t,t}$ is $t$-regular
  and its adjacency spectrum is $\spec(A)=\{t,-t,0^{(2t-2)}\}$.
  Hence $A_\alpha=t\alpha I+(1-\alpha)A$ has spectrum
  $\spec(A_\alpha)=\{t,(2\alpha-1)t,(\alpha t)^{(2t-2)}\}$.
  Therefore $|x_2|=|x_3|=\dots=|x_{2t}|$ holds if and only if
  $|(2\alpha-1)t|=|\alpha t|$. This is equivalent to $\alpha=1/3$.
\end{proof}

\section*{Declaration of competing interest}
There is no competing interest.

\section*{Data availability}
No data was used for the research described in the article.

\section*{Acknowledgement}
The author thanks everyone who contributed
helpful comments that improved the quality
of this manuscript.


\bibliographystyle{elsarticle-num}

\bibliography{cas-refs}

\end{document}